\documentclass{amsart}
\RequirePackage{fix-cm}
\RequirePackage[l2tabu, orthodox]{nag}
\RequirePackage[all, warning]{onlyamsmath}

\usepackage{latexsym}
\usepackage{amsmath,amsfonts,amssymb}
\usepackage{amsthm}
\usepackage{graphicx,color,xcolor}
\usepackage{overpic}
\usepackage{subcaption}
\usepackage{braket}
\usepackage{mathrsfs}
\usepackage[bookmarks=true,bookmarksnumbered=true]{hyperref}
\usepackage{multirow,bigdelim}
\usepackage{enumerate}
\usepackage{cleveref}
\usepackage{comment}
\usepackage{breqn}
\numberwithin{equation}{section}
\usepackage{doi}\usepackage{hyperref}
\newtheorem{theorem}{Theorem}[section]
\newtheorem{lemma}[theorem]{Lemma}
\newtheorem{corollary}[theorem]{Corollary}
\newtheorem{proposition}[theorem]{Proposition}

\theoremstyle{definition}

\newtheorem{remark}[theorem]{Remark}
\crefname{section}{Section}{Sections}
\crefname{appendix}{Appendix}{Appendices}
\crefname{theorem}{Theorem}{Theorems}
\crefname{lemma}{Lemma}{Lemmas}
\crefname{corollary}{Corollary}{Corollaries}			
\crefname{proposition}{Proposition}{Propositions}	
\crefname{claim}{Claim}{Claims}
\crefname{conjecture}{Conjecture}{Conjectures}			
\crefname{definition}{Definition}{Definitions}
\crefname{problem}{Problem}{Problems}
\crefname{example}{Example}{Examples}
\crefname{remark}{Remark}{Remarks}
\crefname{figure}{Figure}{Figures}
\crefname{footnote}{Footnote}{Footnotes}
\crefname{equation}{}{}
\crefname{enumi}{}{}
\crefformat{enumi}{(#2#1#3)}
\Crefformat{enumi}{(#2#1#3)}

\newcommand{\QED}{\hfill \ensuremath{\Box}}

\newcommand{\R}{\mathbb{R}}
\newcommand{\Q}{\mathbb{Q}}

\newcommand{\N}{\mathbb{N}}

\newcommand{\ld}{,\ldots,}

\newcommand{\ep}{\varepsilon}

\newcommand{\norm}[1]{\left\|#1\right\|}

\newcommand{\D}{\displaystyle}

\newfont{\bg}{cmr9 scaled\magstep2}
\newcommand{\bigzerol}{\smash{\lower1.0ex\hbox{\bg 0}}}
\newcommand{\al}{\alpha}

\DeclareMathOperator{\rank}{rank}

\DeclareMathOperator{\id}{id}
\DeclareMathOperator{\supp}{supp}

\title[
Non-density of stable mappings on non-compact manifolds
]
{
Non-density of stable mappings 
\\ 
on non-compact manifolds
}
\author{Shunsuke Ichiki
}
\address{
Department of Mathematical and Computing Science,
School of Computing,
Tokyo Institute of Technology,
Tokyo 152-8552,
Japan}
\email{ichiki@c.titech.ac.jp}

\begin{document}
\date{}
%%%%%%%%%%%%%%%%%% abstract %%%%%%%%%%%%%%%%%%%%%
\begin{abstract}
Around 1970, Mather established a significant theory on the stability of $C^\infty$ mappings and gave a characterization of the density of proper stable mappings in the set of all proper mappings.
The result yields a characterization of the density of stable mappings in the set of all mappings in the case where the source manifold is compact.
The aim of this paper is to complement Mather's result.
Namely, we show that the set of stable mappings in the set of all mappings is never dense if the source manifold is non-compact.
Moreover, as a corollary of Mather's result and the main theorem of this paper, we give a characterization of the density of stable mappings in the set of all mappings in the case where the source manifold is not necessarily compact.
\end{abstract}
%58K25 Stability theory for manifolds
%58K30 Global theory of singularities
%57R35 Differentiable mappings
%57R45 Singularities of differentiable mappings
%%%%%%%%%%%%%%%%%%%%%%%%%%%%%%%%%%%%%%% 
\maketitle
\noindent

%%%%%%%%%%%%%%%%% Introduction %%%%%%%%%%%%%%%
\section{Introduction}\label{sec:intro}
%%%%%%%%%%%%%%%%%%%%%%%%%%%%%%%%%%%%%%%
In the middle of the twentieth century, Whitney conjectured that the set consisting of all stable mappings would always be dense in the appropriate space of $C^\infty$ mappings, and this conjecture came to be known as the ``strong conjecture'' (see for example  \cite{Levine1971}).
However, Thom showed that the set is not necessarily dense in all pairs of dimensions of manifolds. 
Then, around 1970, in a celebrated series  \cite{Mather1968,Mather1968b,Mather1969,Mather1969b,Mather1970,Mather1971}, Mather established a significant theory on the stability of $C^\infty$ mappings and gave a characterization of the density of proper stable mappings in the set of all proper mappings (see \cref{thm:mather}).

In what follows, 
unless otherwise stated, all manifolds and mappings belong to class $C^\infty$,  
and all manifolds are without boundary and assumed to have a countable basis.
%countable bases.
%A mapping is said to be \emph{proper} if the inverse image of any compact subset of the target space is compact.
For manifolds $N$ and $P$, 
we denote the space of all mappings of $N$ into $P$ (resp., the space of all proper mappings of $N$ into $P$) with the Whitney $C^\infty$ topology by $C^\infty(N,P)$ (resp., $C^\infty_{pr}(N,P)$).
For the definition of Whitney $C^\infty$ topology, 
see for example \cite{Golubitsky1973}.
\begin{theorem}[\cite{Mather1971}]\label{thm:mather}
Let $N$ and $P$ be manifolds of dimensions $n$ and $p$, respectively.
Then, the set consisting of all proper stable mappings is dense in $C^\infty_{pr}(N,P)$ if and only if the pair $(n,p)$ satisfies one of the following conditions.
\begin{enumerate}[$(1)$]
\item 
$n<\frac{6}{7}p+\frac{8}{7}$ and $p-n\geq 4$
\item 
$n<\frac{6}{7}p+\frac{9}{7}$ and $3\geq p-n\geq 0$
\item 
$p<8$ and $p-n=-1$
\item 
$p<6$ and $p-n=-2$
\item 
$p<7$ and $p-n\leq-3$
\end{enumerate}
\end{theorem}
A dimension pair $(n,p)$ is called \emph{nice} if $(n,p)$ satisfies one of the conditions (1)--(5) in \cref{thm:mather}.
Note that in \cref{thm:mather} if $N$ is compact, then we have $C^\infty_{pr}(N,P)=C^\infty(N,P)$.
Thus, \cref{thm:mather} yields a characterization of the density of stable mappings in $C^\infty(N,P)$ in the case where $N$ is compact.

After that, the case where a source manifold is non-compact was considered by Dimca, and in 1979,
he gave the following result.
\begin{proposition}[\cite{Dimca1979}]\label{thm:dimca}
Let $N$ be a non-compact manifold.
Then, the set consisting of all stable mappings is not dense in $C^\infty(N,\R)$.
Moreover, the set consisting of all infinitesimally stable mappings is not dense in $C^\infty(N,\R)$.
\end{proposition}
For the definition of infinitesimal stability, which is defined by Mather in \cite{Mather1969}, see \cref{sec:preparation} (the definition of stability is also reviewed in this section).
The main purpose of this paper is to give a rigorous proof of the following main theorem.
\begin{theorem}\label{thm:main}
Let $N$ be a non-compact manifold, and $P$ a manifold.
Then, the set consisting of all stable mappings is not dense in $C^\infty(N,P)$.
Moreover, the set consisting of all infinitesimally stable mappings is not dense in $C^\infty(N,P)$.
\end{theorem}

As a corollary of Mather's theorem (\cref{thm:mather}) and \cref{thm:main}, we easily obtain a characterization of the density of stable mappings in $C^\infty(N,P)$ in the case where $N$ is not necessarily compact as follows:
\begin{corollary}\label{thm:combine}
Let $N$ and $P$ be manifolds of dimensions $n$ and $p$, respectively.
Then, the set consisting of all stable mappings is dense in $C^\infty(N,P)$ if and only if $N$ is compact and $(n,p)$ is nice.
\end{corollary}

The remainder of this paper is organized as follows.
In \cref{sec:preparation}, we prepare some definitions and notations, and give a lemma for the proof of the main theorem (\cref{thm:main}).
In \cref{sec:proof}, we show \cref{thm:main}.

%%%%%%%%%%%%%%%%%%%%%%%%%%%%%%%%%%%%%%%%%%%%%%%%%% 
\section{Preliminaries}\label{sec:preparation}
In this section, we prepare some definitions and notations, and give a lemma (\cref{thm:critical}) for the proof of \cref{thm:main}. 
Moreover, as a supplement, we give a remark on the stability and the infinitesimal stability.

Let $N$ and $P$ be manifolds.
For given mappings $f, g\in C^\infty(N,P)$, 
we say that $f$ is \emph{$\mathcal{A}$-equivalent} to $g$ 
if there exist diffeomorphisms $\Phi:N\to N$ and 
$\Psi:P\to P$ such that $f=\Psi \circ g \circ \Phi^{-1}$.
Let $f:N\to P$ be a mapping.
We say that $f$ is \emph{stable} if the $\mathcal{A}$-equivalence class of $f$ is open in $C^\infty(N,P)$. 
A mapping $\xi:N\to TP$ such that $\Pi \circ \xi=f$ is called a \emph{vector field along $f$}, 
where $TP$ is the tangent bundle of $P$ and $\Pi : TP\to P$ is the canonical projection. 
Let $\theta(f)$ be the set consisting of all vector fields along $f$. 
Set $\theta(N)=\theta(\id_N)$ and $\theta(P)=\theta(\id_P)$, 
where $\id_N:N\to N$ and $\id_P:P\to P$ 
are the identify mappings. 
Following Mather, let $tf:\theta(N)\to \theta(f)$ (resp., $\omega f:\theta(P)\to \theta(f)$) be the mapping  defined by $tf(\xi)=Tf \circ \xi$ (resp., $\omega f(\eta) =\eta \circ f$), 
where $Tf:TN\to TP$ is the derivative mapping of $f$.
Then, as in \cite{Mather1969}, $f$ is said to be \emph{infinitesimally stable} 
if 
\begin{align*}
tf(\theta(N))+\omega f(\theta(P))=\theta(f).
\end{align*}

Let $N$ and $P$ be manifolds, and let $f : N\to P$ be a mapping.
A point $q\in N$ is called a \emph{critical point} of $f$ if $\rank df_q<\dim P$.
We say that a point of $P$ is a \emph{critical value} if it is the image of a critical point.

In what follows, for a given positive integer $m$, we denote the origin $(0\ld 0)$ of $\R^m$ by $0$, the Euclidean norm of $x\in \R^m$ by $\|x\|$, and the $m$-dimensional open ball with center $x\in \R^m$ and radius $r>0$ by $B^m(x,r)$, that is,
\begin{align*}
B^m(x,r)=\set{x'\in \R^m|\|x-x'\|<r}.
\end{align*}
For a set (resp. a topological space) $X$ and a subset $A$ of $X$, we denote the complement of $A$ (resp. the closure of $A$) by $A^c$ (resp. $\overline{A}$).
We denote the set of all positive integers by $\N$.
\begin{lemma}\label{thm:critical}
Let $f=(f_1\ld f_p):B^n(0,r)\to \R^p$ $(r>0)$ be a mapping such that 
\begin{align*}
    f_p(x)=\frac{1}{2}\sum_{i=1}^nx_i^2+a,
\end{align*} where $a$ is a real number and $x=(x_1\ld x_n)$.
If $g=(g_1\ld g_p):B^n(0,r)\to \R^p$ satisfies that 
\begin{align}\label{eq:d0}
\sqrt{\sum_{i=1}^n\left(\frac{\partial f_p}{\partial x_i}(x)-
    \frac{\partial g_p}{\partial x_i}(x)\right)^2}
    <\frac{r}{2}
\end{align}
for any $x\in B^n(0,r)$, then there exists a critical point of $g$ in $B^n(0,r)$.
\end{lemma}
\begin{proof}[Proof of \cref{thm:critical}]
Since in the case $n<p$ any point in $B^n(0,r)$ is a critical point of $g$, it is sufficient to consider the case $n\geq p$.
Let $D:B^n(0,r)\to \R^n$ be the mapping given by 
\begin{align*}
D(x):=&\left(\frac{\partial f_p}{\partial x_1}(x)-
    \frac{\partial g_p}{\partial x_1}(x)\ld 
    \frac{\partial f_p}{\partial x_n}(x)-
    \frac{\partial g_p}{\partial x_n}(x)
    \right)
    \\ 
    =&\left(x_1-
    \frac{\partial g_p}{\partial x_1}(x)\ld 
    x_n-
    \frac{\partial g_p}{\partial x_n}(x)
    \right).
\end{align*}
For simplicity, set $K=\overline{B^n(0,\frac{r}{2})}$.
Since $\norm{D(x)}<\frac{r}{2}$ for any $x\in K$ by \cref{eq:d0}, we can define the restriction $D|_K:K\to K$.
Since $D|_K$ is continuous, there exists a point $x_0\in  K$ such that $D|_K(x_0)=x_0$ by Brouwer’s fixed point theorem.
Thus, it follows that $\frac{\partial g_p}{\partial x_i}(x_0)=0$ for any $i\in \set{1\ld n}$, which implies that $x_0$ is a critical point of $g$.
\end{proof}

\begin{remark}\label{rem:exp}
As in \cite{Mather1970}, for a proper mapping, the stability and the infinitesimal stability are equivalent conditions.
On the other hand, note that in general, they are not equivalent conditions as follows.
First, Mather gives a non-proper mapping which is infinitesimally stable but not stable in \cite[p. 267]{Mather1969}.
Moreover, the non-proper function $f:\R\to \R$ defined by $f(x)=e^{-x}\sin x$ is stable but not infinitesimally stable\footnote{In \cite{Dimca1979}, it is shown that a given function $g:\R\to \R$ is stable if and only if $g$ is a Morse function (i.e., any critical point of $g$ is nondegenerate) with distinct critical values satisfying that $\Delta$ is discrete, $\lim_{x\to \infty}g(x)\not\in \Delta$ if the limit exists and $\lim_{x\to -\infty}g(x)\not\in \Delta$ if the limit exists, where $\Delta$ is the set of all critical values of $g$. %(see also \cite[p.  109]{Wall1995}).
The stability of $f(x)=e^{-x}\sin x$ follows from this result.
On the other hand, since $f|_\Sigma:\Sigma\to \R$ is not proper, $f$ is not infinitesimally stable by \cite[Proposition~5.1]{Mather1970}, where $\Sigma$ is the set of all critical points of $f$.}.
\end{remark}

%%%%%%%%%%%%%%%%%%%%%%%%%%%%%%%%%%%%%%%%%%%%%%%%%%    
\section{Proof of the main theorem}\label{sec:proof}
%%%%%%%%%%%
%%%%%%%%%%%%%%%%%%%%%%%%%%%%%%%%%%
The proof is separated by four steps as follows:
In STEP~1, we construct the $C^\infty$ mapping $f:N\to P$ defined by \eqref{eq:f}.
In STEP~2, we construct the open neighborhood $\mathcal{U}$ of $f$ in $C^\infty(N,P)$ defined by \eqref{eq:u}, and we give a lemma on properties of a mapping in $\mathcal{U}$ (see \cref{thm:u}).
In STEP~3, we show that any mapping in $\mathcal{U}$ is not infinitesimally stable.
Finally, in STEP~4, we prove that any mapping in $\mathcal{U}$ is not stable by preparing two lemmas (\cref{thm:contain,thm:summary}).

\smallskip 
\underline{STEP~1}.
Set $n=\dim N$ and $p=\dim P$, respectively.
By Whitney's embedding theorem, there exist a positive integer $\ell$ and an embedding $F:N\to \R^\ell$ such that $F(N)$ is a closed set of $\R^\ell$.
By taking $\ell$ larger if necessary we can  assume that $F(N)\not=\R^\ell$.
Then, there exists a point $z_0\in \R^\ell\setminus F(N)$.
Since $N$ is non-compact, $F(N)$ is also non-compact.
Thus, $F(N)$ is not bounded, which implies that there exists a sequence $\set{R_\al}_{\al\in \N}$ of positive real numbers and a sequence $\set{z_\al}_{\al\in \N}$ of points in $\R^\ell$ such that 
\begin{itemize}
    \item 
    $R_\al<R_{\al+1}$ for any $\al\in \N$ and $\D\lim_{\al \to \infty}R_\al=\infty$,
    \item
    $z_\al\in F(N)\cap (B^\ell(z_0,R_{\al+1})\setminus \overline{B^\ell(z_0,R_{\al})})$ for any $\al\in \N$.
\end{itemize}
Let $\al$ be any positive integer.
Set $q_\al=F^{-1}(z_\al)$.
Here, note that $$F^{-1}(B^\ell(z_0,R_{\al+1})\setminus \overline{B^\ell(z_0,R_{\al})})$$ is an open neighborhood of $q_\al$.
Then, there exists a coordinate neighborhood $(U_\al,\varphi_\al)$ of $N$ with the following properties:
\begin{itemize}
    \item 
    $\overline{U_\al}$ is compact,
    \item
    $q_\al\in U_\al\subset  F^{-1}(B^\ell(z_0,R_{\al+1})\setminus \overline{B^\ell(z_0,R_{\al})})$,
    \item
    $\varphi_\al(q_\al)=0\in \R^n$.
\end{itemize}
Moreover, there exist an open neighborhood $U_\al'$ of $q_\al$ and $\rho_\al:N\to \R$ such that 
\begin{itemize}
    \item 
    $\overline{U_\al'}\subset U_{\al}$,  
    \item
    $\rho_\al(q)=1$ for any $q\in \overline{U_\al'}$,
    \item
    $\supp \rho_\al\subset U_\al$,
\end{itemize}
where $\supp \rho_\al=\overline{\set{q\in N|\rho_\al(q)\not=0}}$.
Note that $\supp \rho_\al$ is compact since $\overline{U_\al}$ is compact.
Here, by choosing $U_\al'$ smaller for each $\al\in \N$ we can assume that 
\begin{itemize}
    \item 
    $\varphi_\al(U_\al')=B^n(0,r_\al)$,
    \item 
    $\D\lim_{\al \to \infty}r_\al=0$,
\end{itemize}
where each $r_\al$ is a positive real number.

Let $\gamma=(\gamma_1\ld \gamma_p):\N\to \Q^p$ be a bijection, and let $\eta_\al:\varphi_\al(U_\al)\to \R^p$ be the mapping defined by 
\begin{align*}
\eta_\al(x)=
\D\left(\gamma_1(\al)\ld \gamma_{p-1}(\al),\frac{1}{2}\sum_{i=1}^nx_i^2+\gamma_{p}(\al)\right)
\end{align*}
for each $\al\in \N$, where $x=(x_1\ld x_n)$.
Let $(V,\psi)$ be a coordinate neighborhood of $P$ satisfying $\psi(V)=\R^p$.
Since $U_\al\cap U_\beta=\varnothing$ if $\al\not=\beta$, we can define $f:N\to P$ as follows:
\begin{align}\label{eq:f}
f(q)=\left\{ \begin{array}{ll}
\D\psi^{-1}(\rho_\al(q)(\eta_\al\circ \varphi_\al)(q)) & \mbox{if $q\in U_\al$}, 
\\
\\ 
\D\psi^{-1}(0) & \mbox{if $q\not\in \bigcup_{\al\in \N}U_\al$}.
\end{array} \right.
\end{align}
We show that $f$ is of class $C^\infty$.
Let $q\in N$ be any point.
If $q\in \bigcup_{\al\in \N}U_\al$, then by the definition of $f$ it is clearly seen that $f$ is of class $C^\infty$ at $q$.
Thus, we consider the case $q\in (\bigcup_{\al\in \N}U_\al)^c$.
Since $\D\lim_{\al\to \infty}R_\al=\infty$, there exists $\beta \in\N$ such that $q\in F^{-1}(B^\ell(z_0,R_\beta))$.
For simplicity, set
\begin{align*}
A=F^{-1}(B^\ell(z_0,R_\beta))\cap \left(\bigcup_{\al\in \N}\supp \rho_\al\right)^c.
\end{align*}
Note that $q\in A$.
Since $R_\al<R_{\al+1}$ for any $\al\in \N$, we see that $$F^{-1}(B^\ell(z_0,R_\beta))\subset (\supp \rho_\al)^c$$ for any $\al\in \N$ satisfying $\al>\beta$.
Thus, we have 
\begin{align*}
    A&=F^{-1}(B^\ell(z_0,R_\beta))\cap \left(\bigcap_{\al\in \N}(\supp \rho_\al)^c\right)
    \\
    &=F^{-1}(B^\ell(z_0,R_\beta))\cap \left(\bigcap_{\al\leq \beta}(\supp \rho_\al)^c\right),
\end{align*}
which implies that $A$ is an open set of $N$.
Since $\rho_\al |_{A}$ is a constant function with a constant value $0$ for each $\al\in \N$, the mapping $f|_A$ is also constant.
Therefore, $f$ is of class $C^\infty$ at $q$.

\smallskip 
\underline{STEP~2}.
In this step, we construct the open neighborhood $\mathcal{U}$ of $f$ in $C^\infty(N,P)$ defined by \eqref{eq:u}, and we give a lemma on properties of a mapping in $\mathcal{U}$.
Since $z_0\in \R^\ell\setminus F(N)$, we can define the following continuous function $\delta:N\to \R$:
\begin{align*}
    \delta(q)=\frac{1}{\norm{F(q)-z_0}}.
\end{align*}
Let $\pi:J^1(N,P)\to N\times P$ be the natural projection defined by $\pi(j^1g(q))=(q,g(q))$.
Then, for any $\al\in \N$, set 
%{\small 
\begin{align*}
    O_\al=\set{j^1g(q)\in \pi^{-1}(U_\al\times V)|\mbox{$j^1g(q)$ satisfies \cref{eq:t} and  \eqref{eq:d}}},
\end{align*}where
\begin{align}\label{eq:t}
    &\norm{(\psi\circ f)(q)-(\psi\circ g)(q)}<\delta(q),
    \\
    &\sqrt{\sum_{i=1}^n\left(\frac{\partial (\psi_p\circ f\circ \varphi_\al^{-1})}{\partial x_i}(\varphi_\al(q))-
    \frac{\partial (\psi_p\circ g\circ \varphi_\al^{-1})}{\partial x_i}(\varphi_\al(q))\right)^2}<\frac{r_\al}{2}.
    \label{eq:d}
\end{align}
In \eqref{eq:d}, $\psi_p$ is the $p$-th component of $\psi$.
From \cref{eq:t} and \eqref{eq:d}, it is not hard to see that $O_\al$ is an open set of $J^1(N,P)$.

We show that $\bigcap_{\al\in \N}(\overline{U_\al'})^c$ is an open set of $N$.
Let $q\in \bigcap_{\al\in \N}(\overline{U_\al'})^c$ be any point.
Since $\D\lim_{\al\to \infty}R_\al=\infty$, there exists $\beta\in \N$ such that $q\in F^{-1}(B^\ell(z_0,R_\beta))$.
Since $R_\al<R_{\al+1}$ for any $\al\in \N$, we have $F^{-1}(B^\ell(z_0,R_\beta))\subset (\overline{U_\al'})^c$ for any $\al\in \N$ satisfying $\al>\beta$, which implies that 
\begin{align*}
    F^{-1}(B^\ell(z_0,R_\beta))\cap 
    \left(\bigcap_{\al\leq \beta}(\overline{U_\al'})^c\right)\subset \bigcap_{\al\in \N}(\overline{U_\al'})^c.
\end{align*}
Since the left side of the above expression is an open neighborhood of $q$, it follows that $\bigcap_{\al\in \N}(\overline{U_\al'})^c$ is open.
Thus, since $\pi$ is continuous, 
\begin{align*}
    O:=\left(\bigcup_{\al\in \N}O_\al\right)\cup \pi^{-1}\left(\left(\bigcap_{\al\in \N}\left(\overline{U_\al'}\right)^c\right)\times V\right)
\end{align*}
is open in $J^1(N,P)$. 
Therefore, we can construct the following open set of $C^\infty(N,P)$:
\begin{align}\label{eq:u}
    \mathcal{U}:=\set{g\in C^\infty(N,P)|j^1g(N)\subset O}.
\end{align}
By showing that $j^1f(N)\subset O$, we will prove that $\mathcal{U}\not=\varnothing$.
Let $j^1f(q)$ $(q\in N)$ be any element of $j^1f(N)$.
If there exists $\al \in \N$ such that $q\in U_\al$, we have $j^1f(q)\in O_\al$ $(\subset O)$ since $f(q)\in V$ and $j^1f(q)$ clearly satisfies \cref{eq:t} and \eqref{eq:d}.
In the case where $q\not\in \bigcup_{\al\in \N}U_\al$,  since \begin{align}\label{eq:n}
    N=\left(\bigcup_{\al\in \N}U_\al\right) 
    \cup
    \left(\bigcap_{\al\in \N}\left(\overline{U_\al'}\right)^c\right),
\end{align}
it must follow that $q\in \bigcap_{\al\in \N}\left(\overline{U_\al'}\right)^c$.
Therefore, since $f(q)\in V$, we obtain 
\begin{align*}
    j^1f(q)\in \pi^{-1}\left(\left(\bigcap_{\al\in \N}(\overline{U_\al'})^c\right)\times V\right)\ (\subset O).
\end{align*}
Hence, we have $\mathcal{U}\not=\varnothing$.
We give the following lemma on properties of a mapping in $\mathcal{U}$. 
\begin{lemma}\label{thm:u}
For any mapping $g\in \mathcal{U}$, we have $g(N)\subset V$ and there exists a sequence $\set{q_\al'}_{\al\in \N}$ of points in $N$ with the following properties.
\begin{enumerate}[$(1)$]
    \item \label{thm:u_critical}
    For each $\al \in \N$, $q_\al'$ is a critical point of $g$ in $U_\al'$.
    \item \label{thm:u_dense}
    The set $\set{g(q_\al')|\al\in\N}$ is dense in $V$.
\end{enumerate}
\end{lemma}
\begin{proof}[Proof of \cref{thm:u}]
By the definition of $\mathcal{U}$, we have  $g(N)\subset V$.

Let $\al$ be any positive integer.
Then, we have 
\begin{align*}
    (\psi_p\circ f\circ \varphi_\al^{-1})(x)=\frac{1}{2}\sum_{i=1}^nx_i^2+\gamma_p(\al)
\end{align*}
for any $x=(x_1\ld x_n)\in \varphi_\al(U_\al')$ $(=B^n(0,r_\al))$.
For any $q\in U_\al'$, we obtain $j^1g(q)\in O_\al$ since we have \cref{eq:n} and $U_\al'$ is contained in $U_\al$ which does not intersect with $U_\beta$ $(\beta\not=\al)$.
Hence, it follows that $(\psi_p\circ g\circ \varphi_\al^{-1})|_{B^n(0,r_\al)}$ satisfies \eqref{eq:d}, which implies that there exists a critical point of $(\psi_p\circ g\circ \varphi_\al^{-1})|_{B^n(0,r_\al)}$ in $B^n(0,r_\al)$ by \cref{thm:critical}.
Namely, there exists a critical point of $g$ in $U_\al'$.
We denote its point by $q_\al'$.

Since $\set{q_\al'}_{\al\in \N}$ satisfies \cref{thm:u_critical} by the above argument, it is sufficient to show that the sequence of points also satisfies \cref{thm:u_dense}.
Let $V'$ be any open set of $V$.
We show that $\set{g(q_\al')|\al\in \N}\cap V'\not=\varnothing$.
Then, by choosing $V'$ smaller, we can assume that $\psi(V')=B^p(y_0,\ep)$, where $y_0$ is a point of $\R^p$ and $\ep$ is a positive real number.
Note that for any $\al\in \N$, we have 
\begin{equation}
\begin{split}\label{eq:al_0}
    \norm{(\psi\circ g)(q_\al')-y_0}\leq 
    &\norm{(\psi\circ g)(q_\al')-(\psi\circ f)(q_\al')}
    +  \\
    &\norm{(\psi\circ f)(q_\al')-(\psi\circ f)(q_\al)}
  +
    \norm{(\psi\circ f)(q_\al)-y_0}.
   \end{split}
\end{equation}

Since 
\begin{align*}
    \delta(q_\al')=\frac{1}{\norm{F(q_\al')-z_0}}<\frac{1}{R_\al}
\end{align*}
for any $\al\in \N$ and $\D\lim_{\al\to \infty}R_\al=\infty$, there exists $\al_1\in \N$ such that $\delta(q_\al')<\frac{\ep}{3}$ for any $\al\in \N$ satisfying $\al\geq \al_1$.
Here, note that for any $\al\in \N$, we have 
\begin{align*}
        \norm{(\psi\circ g)(q_\al')-(\psi\circ f)(q_\al')}<\delta(q_\al')
\end{align*}
by \cref{eq:t} since $j^1g(q_\al')\in O_\al$.
Thus, it follows that for any $\al\in \N$,
\begin{align}\label{eq:al_1}
    \al\geq \al_1 \Longrightarrow  \norm{(\psi\circ g)(q_\al')-(\psi\circ f)(q_\al')}<\frac{\ep}{3}.
\end{align}

For any $\al\in \N$, since $q_\al, q_\al'\in U_\al'$, we have 
\begin{align*}
     \norm{(\psi\circ f)(q_\al')-(\psi\circ f)(q_\al)}=
     \norm{\eta_\al(\varphi_\al(q_\al'))-\gamma(\al)}
     = \frac{\norm{\varphi_\al(q_\al')}^2}{2}
     < \frac{r_\al^2}{2}.
\end{align*}
Since $\D\lim_{\al\to \infty}r_\al=0$, there exists $\al_2\in \N$ such that for any $\al\in \N$,
\begin{align}\label{eq:al_2}
    \al\geq \al_2 \Longrightarrow  \norm{(\psi\circ f)(q_\al')-(\psi\circ f)(q_\al)}<\frac{\ep}{3}.
\end{align}

Since $(\psi \circ f)(q_\al)=\gamma(\al)$ for each $\al\in \N$, we have 
\begin{align*}
    \set{(\psi \circ f)(q_\al)|\al\in\N}=\Q^p.
\end{align*}
Hence, there exists $\al_3\in \N$ such that $\al_3>\max\set{\al_1,\al_2}$ and 
\begin{align}\label{eq:al_3}
    \norm{(\psi\circ f)(q_{\al_3})-y_0}<\frac{\ep}{3}.
\end{align}
Thus, we have $\norm{(\psi\circ g)(q_{\al_3}')-y_0}<\ep$ by \eqref{eq:al_0} to \cref{eq:al_3}, which implies that $g(q_{\al_3}')\in V'$.
\end{proof}

\smallskip 
\underline{STEP~3}.
The purpose of this step is to show that any mapping in $\mathcal{U}$ is not infinitesimally stable.
Let $g\in \mathcal{U}$ be any mapping, and let $\Sigma$ be the set consisting of all critical points of $g$.
Set $K=\psi^{-1}(\overline{B^p(0,r)})$, where $r$ is a positive real number.
Note that $K$ is a compact set in $P$.
Then, from \cref{thm:u}~\cref{thm:u_dense}, $(g|_\Sigma)^{-1}(K)$ contains a countable subset of $\set{q_\al'|\al\in \N}$.
Since $F(q_\al')\not\in  \overline{B^\ell(z_0,R_\al)}$ for each $\al \in \N$ and $\D\lim_{\al\to \infty}R_\al=\infty$, the set $F((g|_\Sigma)^{-1}(K))$ is not compact, which implies that $(g|_\Sigma)^{-1}(K)$ is not compact.
Since $g|_\Sigma:\Sigma\to P$ is not proper, $g$ is not infinitesimally stable (note that this fact follows from \cite[Proposition~5.1]{Mather1970}).

\smallskip 
\underline{STEP~4}.
The purpose of this step is to show that any mapping in $\mathcal{U}$ is not stable.
Let $g\in \mathcal{U}$ be an arbitrary mapping, and let $\mathcal{U}_g$ be any open neighborhood of $g$.
Then, there exist a non-negative integer $k$ and an open set $O'$ of $J^k(N,P)$ such that 
\begin{align*}
    g\in \set{h\in C^\infty(N,P)|j^kh(N)\subset O'}\subset \mathcal{U}_g.
\end{align*}
In order to prove that $g$ is not stable, it is sufficient to show that there exists a mapping $h\in C^\infty(N,P)$ satisfying the following properties.
\begin{itemize}
    \item \label{thm:contain_jet} We have $j^kh(N)\subset O'$.
    \item \label{thm:not_stable}There exist $(p+1)$-critical points of $h$ which share the same
critical value.
\end{itemize}
Note that the second property implies that $h$ is not stable.

For any $\al\in \N$ and $c\in \R^p$, let $G_{\al,c}:N\to P$ be the mapping defined by 
\begin{align*}
    G_{\al,c}=\psi^{-1}\circ (\psi \circ g+\rho_\al c).
\end{align*}
\begin{lemma}\label{thm:contain}
Let $\al$ be any positive integer.
Then, there exists a positive real number $r_\al'$ such that $j^kG_{\al,c}(N)\subset O'$ for any $c\in B^p(0,r_\al')$.
\end{lemma}
\begin{proof}[Proof of \cref{thm:contain}]
Let $\Gamma_\al:N\times \R^p\to J^k(N,P)$ be the mapping defined by $\Gamma_\al(q,c)=j^kG_{\al,c}(q)$.
For any $q\in \supp \rho_\al$, since $\Gamma_\al(q,0)=j^kg(q)\in O'$ and $\Gamma_\al$ is continuous at $(q,0)$, there exist an open neighborhood $U_q$ of $q\in N$ and an open neighborhood $W_q$ of $0\in \R^p$ such that $\Gamma_\al(U_q\times W_q)\subset O'$.
Since $\set{U_q}_{q\in \supp \rho_\al}$ is an open covering of the compact set $\supp \rho_\al$, there exists a finite subset $S$ of $\supp \rho_\al$ such that $\supp \rho_\al\subset \bigcup_{q\in S}U_q$.
Since $\bigcap_{q\in S}W_q$ is an open neighborhood of $0\in \R^p$, there exists a positive real number $r_\al'$ such that $B^p(0,r_\al')\subset \bigcap_{q\in S}W_q$.

Let $c\in B^p(0,r_\al')$ and $q\in N$ be any points.
If $q\not\in \supp \rho_\al$, then we have $j^kG_{\al,c}(q)\in O'$ since $G_{\al,c}=g$ on the open neighborhood $(\supp \rho_\al)^c$ of $q$.
If $q\in \supp \rho_\al$, then there exists a point $q_0\in S$ such that $q\in U_{q_0}$.
Since $c\in W_{q_0}$, we obtain $j^kG_{\al,c}(q)=\Gamma_\al(q,c)\in O'$.
\end{proof}
Since $g\in \mathcal{U}$, note that there exists a sequence $\set{q_\al'}_{\al\in \N}$ of points in $N$ satisfying \cref{thm:u_critical} and \cref{thm:u_dense} of \cref{thm:u}.
\begin{lemma}\label{thm:summary}
Let $m$ be any positive integer.
Then, there exist $(m+1)$ distinct positive integers $\al_1\ld \al_{m+1}$ and $m$ positive real numbers $r_{\al_1}'\ld r_{\al_m}'$ $(r_{\al_1}'>\cdots >r_{\al_m}')$ such that for any $j\in \set{1\ld m}$,
\begin{enumerate}[$(1)$]
    \item 
    $j^kG_{\al_j,c}(N)\subset O'$ for any $c\in B^p(0,r_{\al_j}')$,
    \item 
    $\norm{(\psi\circ g)(q_{\al_{j+1}}')-(\psi\circ g)(q_{\al_{j}}')}<\D\frac{r_{\al_{j}}'}{p}$.
\end{enumerate}
\end{lemma}
\begin{proof}[Proof of \cref{thm:summary}]
We prove the lemma by induction on $m$.

Let $\al_1$ be any positive integer.
By \cref{thm:contain}, there exists a positive real number $r_{\al_1}'$ such that $j^kG_{\al_1,c}(N)\subset O'$ for any $c\in B^p(0,r_{\al_1}')$.
By \cref{thm:u}~\cref{thm:u_dense}, there exists $\al_2 \in \N\setminus\set{\al_1}$ such that $$\norm{(\psi\circ g)(q_{\al_2}')-(\psi\circ g)(q_{\al_1}')}<\frac{r_{\al_1}'}{p}.$$
Hence, the case $m=1$ holds.

We assume that the lemma is true for $m=i$, where $i$ is a positive integer.
By \cref{thm:contain}, there exists a positive real number $r_{\al_{i+1}}'$ $(r_{\al_i}'>r_{\al_{i+1}}')$ such that $j^kG_{\al_{i+1},c}(N)\subset O'$ for any $c\in B^p(0,r_{\al_{i+1}}')$.
By \cref{thm:u}~\cref{thm:u_dense}, there exists $\al_{i+2}\in \N\setminus\set{\al_1\ld \al_{i+1}}$ such that $\norm{(\psi\circ g)(q_{\al_{i+2}}')-(\psi\circ g)(q_{\al_{i+1}}')}<\frac{r_{\al_{i+1}}'}{p}$.
Therefore, the case $m=i+1$ holds.
\end{proof}
For simplicity, set $I=\set{1\ld p}$.
By \cref{thm:summary} in the case $m=p$, there exist $(p+1)$ distinct positive integers $\al_1\ld \al_{p+1}$ and $p$ positive real numbers $r_{\al_1}'\ld r_{\al_p}'$ $(r_{\al_1}'>\cdots >r_{\al_p}')$ such that for any $j\in I$,
\begin{enumerate}[(a)]
    \item \label{thm:summary_contain}
    $j^kG_{\al_j,c}(N)\subset O'$ for any $c\in B^p(0,r_{\al_j}')$,
    \item \label{thm:summary_i}
    $\norm{(\psi\circ g)(q_{\al_{j+1}}')-(\psi\circ g)(q_{\al_{j}}')}<\D\frac{r_{\al_{j}}'}{p}$.
\end{enumerate}
Let $h:N\to P$ be the mapping defined by 
\begin{align*}
    h=\psi^{-1}\circ \left(\psi\circ g+\sum_{i=1}^p\rho_{\al_i}c_i\right),
\end{align*}
where $c_i=(\psi\circ g)(q_{\al_{p+1}}')-(\psi\circ g)(q_{\al_i}')\in \R^p$.

First, we show that $j^kh(N)\subset O'$.
Let $q\in N$ be an arbitrary point.
In the case where $q$ is an element of $(\bigcup_{j=1}^p\supp \rho_{\al_j})^c$, since $h=g$ on the open neighborhood $(\bigcup_{j=1}^p\supp \rho_{\al_j})^c$ of $q$, we have $j^kh(q)=j^kg(q)\in O'$.
We consider the case where there exists $j\in I$ such that $q\in \supp \rho_{\al_j}$.
Since $\supp \rho_{\al_j}\subset \bigcap_{i\in I\setminus \set{j}}(\supp \rho_{\al_i})^c$ and $h=G_{\al_j,c_j}$ on the open neighborhood  $\bigcap_{i\in I\setminus \set{j}}(\supp \rho_{\al_i})^c$ of $q$, we have $j^kh(q)=j^kG_{\al_j,c_j}(q)$.
Moreover, since 
\begin{equation}
\begin{split}\label{eq:c}
    \norm{c_j}&=\norm{(\psi\circ g)(q_{\al_{p+1}}')-(\psi\circ g)(q_{\al_{j}}')}
    \\
    &\leq\sum_{i=j}^p\norm{(\psi\circ g)(q_{\al_{i+1}}')-(\psi\circ g)(q_{\al_{i}}')}
    \\
    &<\sum_{i=j}^p\frac{r_{\al_{i}}'}{p}
    \\
    &\leq r_{\al_{j}}',
\end{split}
\end{equation}
we have $c_j\in B^p(0,r_{\al_j}')$.
Note that the last two inequalities in \cref{eq:c} follow from  \cref{thm:summary_i} and the fact that $r_{\al_j}'>\cdots >r_{\al_p}'$, respectively.
Thus, we obtain $j^kG_{\al_j,c_j}(q)\in O'$ by \cref{thm:summary_contain}, which implies that $j^kh(q)\in O'$.

Finally, we show that there exist $(p+1)$-critical points of $h$ which share the same critical value.
For any $i,j\in I$, since $\rho_{\al_i}(q_{\al_j}')=\delta_{ij}$ and $\rho_{\al_i}(q_{\al_{p+1}}')=0$, we obtain{\small
\begin{align*}
    \left(\psi\circ g+\sum_{i=1}^p\rho_{\al_i}c_i\right)(q_{\al_j}')
    =(\psi\circ g)(q_{\al_j}')+c_j
    =(\psi\circ g)(q_{\al_{p+1}}')
    =(\psi\circ h)(q_{\al_{p+1}}'),
\end{align*}}where $\delta_{ij}$ is the Kronecker delta.
Thus, we have $h(q_{\al_1}')=\cdots= h(q_{\al_{p+1}}')$.
Moreover, for any $j\in I$, the point $q_{\al_j}'$  (resp., $q_{\al_{p+1}}'$) is a critical point of $h$ since $h=\psi^{-1}\circ (\psi\circ g+c_j)$ on an open neighborhood of $q_{\al_j}'$ (resp., $h=g$ on an open neighborhood of $q_{\al_{p+1}}'$).
Namely, $q_{\al_{1}}'\ld q_{\al_{p+1}}'$ share the same critical value of $h$.
\QED

%%%%%%%%%%%%%%%    
\section*{Acknowledgements}
The author is most grateful to the anonymous reviewers for carefully reading the first manuscript of this paper and for giving invaluable suggestions.
The author would like to thank Kenta Hayano and Takashi Nishimura for helpful discussions. 
This work was supported by JSPS KAKENHI Grant Number JP21K13786.
%%%%%%%%%%%%%%%%%%%%%%%%%%%%%%%%%%%%%%%%%%%%%%%%%% 
%\bibliographystyle{plain}
%\bibliography{main}

\begin{thebibliography}{99}

\bibitem{Dimca1979}
Alexandru Dimca. 
\newblock Morse functions and stable mappings.
\emph{Rev. Roumaine Math. Pures Appl.,} 24(9):1293–1297, 1979.


\bibitem{Golubitsky1973}
Martin Golubitsky and Victor Guillemin. 
\newblock \emph{Stable Mappings and Their Singularities}, volume 14
of \emph{Graduate Texts in Mathematics.} Springer, New York, 1973.

\bibitem{Levine1971}
Harold I Levine. 
\newblock Singularities of differentiable mappings.
\emph{Proceedings of Liverpool Singularities - Symposium I, Springer, Berlin, Heidelberg}, 192:1–89, 1971.


\bibitem{Mather1968}
John Norman Mather. 
\newblock Stability of {$C^\infty$} mappings. I. The division theorem. 
\emph{Ann. of Math.} (2), 87:89--104, 1968.

\bibitem{Mather1968b}
John Norman Mather. 
\newblock Stability of {$C^\infty$} mappings. III. Finitely determined mapgerms.
\emph{Inst. Hautes \'{E}tudes Sci. Publ. Math.}, 35:279--308, 1968.

\bibitem{Mather1969}
John Norman Mather. 
\newblock Stability of {$C^\infty$} mappings. II. Infinitesimal stability implies stability.
\emph{Ann. of Math.} (2), 89:254--291, 1969.


\bibitem{Mather1969b}
John Norman Mather. 
\newblock Stability of $C^\infty$ mappings. IV. Classification of stable germs by $R$-algebras.
\emph{Inst. Hautes \'{E}tudes Sci. Publ. Math.}, (2) 37:223--248, 1969.


\bibitem{Mather1970}
John Norman Mather. 
\newblock Stability of $C^\infty$ mappings. V. Transversality.
\emph{Advances in Math.}, 4:301--336, 1970.


\bibitem{Mather1971}
John Norman Mather. 
\newblock Stability of $C^\infty$ mappings. VI. The nice dimensions.
\emph{Proceedings of Liverpool Singularities-Symposium, {I}. Lecture Notes in Math.}, 192:207--253, 1971.
\end{thebibliography}

\end{document}